\documentclass[12pt]{amsart}
\usepackage{amscd,amsmath,amsthm,amssymb}
\usepackage{color}
\usepackage{amsfonts,amssymb,amscd,amsmath,enumerate,verbatim,enumitem}
\usepackage{lineno}
 \def\NZQ{\mathbb}               
 \def\NN{{\NZQ N}}
 
 \def\ZZ{{\NZQ Z}}

 \def\frk{\mathfrak}               
 
 \def\pp{{\frk p}}

 \def\mm{{\frk m}}
 
 \def\nn{{\frk n}}
 \def\G{{\mathcal G}}

 \def\ab{{\mathbf a}}

 \def\opn#1#2{\def#1{\operatorname{#2}}} 
 \opn\chara{char} \opn\length{\ell} \opn\pd{pd} \opn\rk{rk}
 \opn\projdim{proj\,dim} \opn\injdim{inj\,dim} \opn\rank{rank}
 \opn\depth{depth} \opn\grade{grade} \opn\height{height}
 \opn\embdim{emb\,dim} \opn\codim{codim}
 
 \opn\Tr{Tr} \opn\bigrank{big\,rank}
 \opn\superheight{superheight}\opn\lcm{lcm}
 \opn\trdeg{tr\,deg}
 \opn\reg{reg} \opn\lreg{lreg} \opn\ini{in} \opn\lpd{lpd}
 \opn\size{size} \opn\sdepth{sdepth}
 \opn\link{link}\opn\fdepth{fdepth}\opn\lex{lex}
 \opn\tr{tr}
 \opn\div{div} \opn\Div{Div} \opn\cl{cl} \opn\Cl{Cl}
 \opn\Spec{Spec} \opn\Supp{Supp} \opn\supp{supp} \opn\Sing{Sing}
 \opn\Ass{Ass} \opn\Min{Min}\opn\Mon{Mon}
 \opn\Ann{Ann} \opn\Rad{Rad} \opn\Soc{Soc}
 \opn\Im{Im} \opn\Ker{Ker} \opn\Coker{Coker} \opn\Am{Am}
 \opn\Hom{Hom} \opn\Tor{Tor} \opn\Ext{Ext} \opn\End{End}
 \opn\Aut{Aut} \opn\id{id}
 
 \opn\nat{nat}
 \opn\pff{pf}
 \opn\Pf{Pf} \opn\GL{GL} \opn\SL{SL} \opn\mod{mod} \opn\ord{ord}
 \opn\Gin{Gin} \opn\Hilb{Hilb}\opn\sort{sort}
 \opn\PF{PF}\opn\Ap{Ap}
 \opn\aff{aff} \opn
 \con{conv} \opn\relint{relint} \opn\st{st}
 \opn\lk{lk} \opn\cn{cn} \opn\core{core} \opn\vol{vol}  \opn\inp{inp} \opn\nilpot{nilpot}
 \opn\link{link} \opn\star{star}\opn\lex{lex}\opn\set{set}
 \opn\width{wd}
 \opn\Fr{F}
 \opn\QF{QF}
 \opn\G{G}
 \opn\type{type}\opn\res{res}
 \opn\gr{gr}
 
 \def\pot#1#2{#1[\kern-0.28ex[#2]\kern-0.28ex]}

 \opn\dirlim{\underrightarrow{\lim}}
 \opn\inivlim{\underleftarrow{\lim}}
 \let\iso=\cong

 \let\to=\rightarrow
 
 \def\Implies{\ifmmode\Longrightarrow \else
         \unskip${}\Longrightarrow{}$\ignorespaces\fi}
 \def\implies{\ifmmode\Rightarrow \else
         \unskip${}\Rightarrow{}$\ignorespaces\fi}
 \def\iff{\ifmmode\Longleftrightarrow \else
         \unskip${}\Longleftrightarrow{}$\ignorespaces\fi}

 \let\:=\colon
 \newtheorem{Theorem}{Theorem}[section]
 \newtheorem{Lemma}[Theorem]{Lemma}
 \newtheorem{Corollary}[Theorem]{Corollary}
 \newtheorem{Proposition}[Theorem]{Proposition}
 \newtheorem{Remark}[Theorem]{Remark}
 
 \newtheorem{Example}[Theorem]{Example}

 \newtheorem{Question}[Theorem]{Question}
 \let\epsilon\varepsilon
 \let\kappa=\varkappa
 \def\qed{\ifhmode\textqed\fi
       \ifmmode\ifinner\quad\qedsymbol\else\dispqed\fi\fi}
 \def\textqed{\unskip\nobreak\penalty50
        \hskip2em\hbox{}\nobreak\hfil\qedsymbol
        \parfillskip=0pt \finalhyphendemerits=0}
 \def\dispqed{\rlap{\qquad\qedsymbol}}
 
 \opn\dis{dis}
 \def\pnt{{\raise0.5mm\hbox{\large\bf.}}}
 
 \opn\Lex{Lex}

\begin{document}
\title {Canonical trace ideal and residue for numerical semigroup rings}

\author {J\"urgen Herzog, Takayuki Hibi and Dumitru I.\ Stamate}

\address{J\"urgen Herzog, Fachbereich Mathematik, Universit\"at Duisburg-Essen, Campus Essen, 45117
Essen, Germany} \email{juergen.herzog@uni-essen.de}

\address{Takayuki Hibi, Department of Pure and Applied Mathematics, Graduate School of Information Science and Technology,
Osaka University, Toyonaka, Osaka 560-0043, Japan}
\email{hibi@math.sci.osaka-u.ac.jp}

\address{Dumitru I. Stamate, ICUB/Faculty of Mathematics and Computer Science, University of Bucharest, Str. Academiei 14, Bucharest -- 010014, Romania }
\email{dumitru.stamate@fmi.unibuc.ro}

\dedicatory{ }

\begin{abstract}
For a numerical semigroup ring $K[H]$ we study the trace of its canonical ideal. The colength of this ideal is called the residue of $H$. This invariant measures how far is $H$ from being symmetric, i.e. $K[H]$ from being a Gorenstein ring. We remark that the canonical trace ideal contains the conductor ideal, and we study bounds for the residue. 

For $3$-generated numerical semigroups we give explicit formulas for the canonical trace ideal and the residue of $H$. Thus, in this setting we can classify those whose residue is at most one (the nearly-Gorenstein ones), and we show the eventual periodic  behaviour of the residue in a shifted family. 
\end{abstract}

\thanks{}




\maketitle
 
\section*{Introduction}
\label{sec:introd}

Let $(R,\mm, K)$ be a local ring (or a positively graded $K$-algebra) which is Cohen-Macaulay and possesses a canonical module $\omega_R$. In \cite{HHS} the trace ideal of $\omega_R$ is used as a tool to stratify the Cohen-Macaulay rings and  to define the class of nearly Gorenstein rings. We recall that if $N$ is any $R$-module, its trace is the ideal $\tr(N)=\sum_{\varphi\in \Hom_R(N,R)} \varphi(N)$ in $R$. 

The relevance of $\tr(\omega_R)$ (also called the canonical trace ideal of $R$) stems from the fact that it describes the non-Gorenstein locus of the ring $R$.
Namely, by \cite[Lemma 2.1]{HHS}, for any $\pp \in \Spec(R)$, $\pp \supseteq \tr(\omega_R)$ if and only if $R_\pp$ is not a Gorenstein ring. Thus $\tr(\omega_R)=R$ if and only if $R$ is a Gorenstein ring. In \cite{HHS}, the ring $R$ is called nearly Gorenstein when $\tr(\omega_R) \supseteq \mm$. Also, the residue of $R$, denoted $\res(R)$  is defined as the length of the module $R/\tr(\omega_R)$.  Several other invariants for such rings are surveyed in \cite{brennen-et-al}.

In this paper we study bounds, and in small codimension we give exact formulas, for $\res(R)$ when $R$ is the semigroup ring $K[H]$ associated to the numerical semigroup $H$ and the field $K$. This allows to determine the nearly Gorenstein property in some families of semigroups.

We outline the structure of the paper.
First, in Section~\ref{sec:prelim} we  transfer the terminology and notations from rings to the setting of numerical semigroups.
A numerical semigroup $H$ is a subsemigroup of $\NN$ containing $0$
such that the number of gaps $g(H)=|\NN \setminus H|$ is finite. The largest gap (i.e. positive integer not in $H$) is the Frobenius number $\Fr(H)$.
In Proposition \ref{prop:arithmetic-ng} we show that   if $H$ is generated by an arithmetic sequence, then $K[H]$ is nearly Gorenstein.

As a measure of how far  is $K[H]$ from being Gorenstein (equivalently, that $H$ is symmetric, cf. \cite{Kunz}), we introduce the residue of $H$ defined as
$$\res(H)=\dim_K  K[H]/\tr(\omega_{K[H]}).$$
Clearly, $\res(H)=0$ when $H$ is symmetric, and $\res(H)\leq 1$ precisely when $K[H]$ is nearly Gorenstein.
The exponents of the monomials in $\tr(\omega_{K[H]})$ form a semigroup ideal $\tr(H)\subseteq H$.
We note in Proposition \ref{prop:subsets-trace} that if  $H$ is not symmetric, then $\mathcal{C}_H \subseteq \tr(H) \subseteq H\setminus\{0\}$, where  $\mathcal{C}_H$
is the semigroup ideal generated by the elements of $H$ larger than $\Fr(H)$.

This observation gives a first estimate
$$\res(H)\leq n(H):= |\{ x\in H: x< \Fr(H)\}|$$ in Corollary \ref{cor:boundres}.
Examples computed with the NumericalSgps package \cite{Num-semigroup} in GAP \cite{GAP} indicate (Question \ref{que:g-n}) that another bound might also hold:
\begin{equation}
\label{eq:intro}
\res(H) \leq n(H)-g(H).
\end{equation}
This bound is proved to be correct if $K[H]$ is nearly Gorenstein, and also if $H$ is $3$-generated, cf. Proposition \ref{prop:3-bound}.

When $H$ is $3$-generated and not symmetric, the relation ideal $I_H \subset K[x_1, x_2, x_3]$ of $K[H]$ is given by the maximal minors   of the structure matrix of $H$,
which is of the form
\begin{eqnarray}
\label{intro-structure}
A=\left( \begin{array}{ccc} x_1^{a_1} & x_2^{a_2} & x_3^{a_3}\\
x_2^{b_2} & x_3^{b_3} & x_1^{b_1}
 \end{array}\right).
\end{eqnarray}
With this notation we derive  in  Proposition \ref{prop:3-semi-trace-formula} that $$\res(H)=\prod_{i=1}^3 \min\{a_i, b_i\}.$$

Working with the structure matrix of $H$ allows us to parametrize explicitly   the non-symmetric $3$-generated semigroups  $H$ whose trace is at either
end of the interval $[\mathcal{C}_H, H\setminus\{0\}]$, see Theorem \ref{thm:3semi-trace-is-maximal} and Proposition \ref{prop:3semi-trace-is-conductor}.

Example \ref{ex:trace-conductor} shows that $\res(H)$ may take any nonnegative integer value, even if we fix the number of generators of $H$.
Still, once we fix $n_1<\dots<n_e$, the residue of the semigroups in the  shifted family $\{ \langle n_1+j,\dots, n_e+j\rangle \}_{j\geq 0}$ 
seem to change  periodically with $j$, for $j\gg 0$.
This goes in the same direction as a recent number of other results about eventually periodic properties in this shifted family, see \cite{JS}, \cite{Vu}, \cite{HeS}, \cite{St-surveybetti}, \cite{C-all}, \cite{OP}.
Using  \cite{S-3semi}, we prove in Theorem  \ref{thm:3semi-res-periodic}
that given $n_1<n_2<n_3$ and letting $H_j=\langle n_1+j, n_2+j, n_3+j\rangle$ we have $\res(H_j)=\res(H_{j+(n_3-n_1)})$ for all $j\gg 0$.
In this setup, in Corollary \ref{cor:bound-res-shifts} we obtain another upper bound for $\res(H_j)$ when $j\gg 0$, depending on $n_3-n_1$.

In the Appendix we prove the inclusion of the conductor ideal in any trace ideal, and we characterize when the equality holds. This is made in the more general context of extensions of local rings $R\subseteq \widetilde{R}$  with isomorphic residue fields, $\widetilde{R}$ a discrete valuation ring in $Q(R)$ and a finite $R$-module..

\section{The canonical trace ideal of a semigroup, or Rings to semigroups transition}
\label{sec:prelim}

A numerical semigroup $H$ is a submonoid of $\NN$, and unless stated otherwise we assume $|\NN\setminus H| < \infty$.
Say $H$ is minimally generated by $n_1<n_2<\ldots <n_e$ with $e>1$. We write $H=\langle n_1,\ldots,n_e\rangle$. The number $e$ is called the {\em embedding dimension} of $H$ and the number $n_1$ the {\em multiplicity }of $H$.  One always has $n_1\leq e$. We say that $H$ has {\em minimal multiplicity} if $n_1=e$.
In this case, one also says that $H$ has {\em maximal embedding dimension}, cf. \cite{RoSa-book}.

 The elements  in the set $G(H)=\NN\setminus H$ are called the {\em gaps} of $H$.
As $|G(H)|<\infty$,  there exists a largest  integer  $\Fr(H)$, called the {\em Frobenius number} of $H$,  such that $\Fr(H)\not\in H$.

 We denote by $M$ the subset $H\setminus \{0\}$. The elements $f\in G(H)$  with $f+M\in H$ are called {\em pseudo-Frobenius numbers}.
The set of pseudo-Frobenius numbers will be denoted by $\PF(H)$. The cardinality of $\PF(H)$ is called the {\em type} of $H$, denoted $\type(H)$.

 We fix a field $K$. The positively graded $K$-subalgebra $K[H]=K[t^{n_1},\ldots,t^{n_e}]$ of $K[t]$ is the semigroup ring of $H$. Its graded maximal ideal is $\mm= (t^{n_1},\ldots,t^{n_e})$. The  embedding dimension (resp.\ multiplicity) of $H$ is also the  embedding dimension (resp.\ multiplicity) of $K[H]$ in the algebraic sense.
The polynomial ring $K[t]$ is a finite module over $K[H]$ and is the integral closure of $K[H]$ in its quotient field $Q(K[H])=K(t)$. The module $K[t]/K[H]$ has finite length and a  $K$-basis given by the residue classes of $\{t^a\: \; a\in G(H)\}$.


 The canonical module $\omega_{K[H]}$ of $K[H]$  is the fractionary $K[H]$-ideal generated by the elements $t^{-f}$ with $f\in \PF(H)$, see \cite[Exercise 21.11]{Eis}.
Therefore, the Cohen-Macaulay type of $K[H]$ is equal to $\type(H)$. In particular,  $K[H]$ is Gorenstein  if and only if $\PF(H)=\{\Fr(H)\}$.
Kunz \cite{Kunz} showed that $K[H]$ is Gorenstein if and only if $H$ is {\em symmetric}, i.e. for all $x\in \ZZ$ either $x\in H$, or $\Fr(H)-x \in H$.
The anti-canonical ideal of $K[H]$ is the fractionary ideal $\omega^{-1}_{K[H]}=\{x\in Q(K[H]): x\cdot \omega_{K[H]}\subseteq K[H] \}$.
Since $K[H]$ is a domain, by \cite[Lemma 1.1]{HHS} one has $\tr(\omega_{K[H]})=\omega_{K[H]}\cdot \omega_{K[H]}^{-1}$.

We mention that the almost Gorenstein numerical semigroup rings (as defined by Barucci and Fr\"oberg in \cite{BF}, see also \cite{GMT}) are a proper subclass of the nearly Gorenstein ones, by  \cite[Proposition 6.1]{HHS}. For our purposes,  we will take as definition for almost Gorensteinness Nari's characterization which we explain next. 

  Let $\PF(H)=\{f_1,\ldots,f_{\tau-1},\Fr(H)\}$, with $f_i<f_{i+1}$ for $1\leq i <\tau-2$. It is known by Nari \cite{Nari} that $K[H]$ is almost Gorenstein, if and only  if
\begin{equation}
\label{eq:ag-symmetries}
 f_i+f_{\tau-i}=\Fr(H)\quad \text{for}\quad i=1,\ldots, \lfloor \tau/2 \rfloor.
\end{equation}
The semigroup $H$ is called  {\em almost symmetric} if $K[H]$ is almost Gorenstein, and $H$ is called {\em nearly Gorenstein}, if $K[H]$ is nearly Gorenstein. These two classes of semigroups have been recently considered in \cite{MoSt}.

A subset $I\subset \ZZ$ is called a {\em relative ideal} of $H$ if $I+H \subseteq I$ and $h+I\subseteq H$ for some $h\in H$.
 If moreover $I\subseteq H$, then $I$ is called an {\em ideal} of $H$.

Let $\Omega_H$ and $\Omega_H^{-1}$ be the set  of exponents of the monomials in $\omega_{K[H]}$, and in $\omega_{K[H]}^{-1}$ respectively.
Then $\Omega_H$ and $\Omega_H^{-1}$ are relative  ideals of $H$ called  the canonical, respectively the anti-canonical ideal of $H$.
We define the {\em trace} of $H$ as $\tr(H)= \Omega_H+ \Omega_H^{-1}$.
It is clear that  $\tr(H)$  is an ideal in $H$  consisting  of the exponents of  the monomials in $\tr(K[H])$.

 In this notation,  $H$ is nearly Gorenstein if and only if $M\subseteq \tr(H)$.

\medskip
The semigroup ring $K[H]$ is $1$-dimensional, so its canonical trace ideal is either the whole ring, or it is an $\mm$-primary ideal.
Equivalently, $K[H]/\tr(\omega_{K[H]})$ is a finite dimensional vector space with a $K$-basis given by  $\{t^h: h\in H\setminus \tr(H) \}$.
We define the {\em residue} of $H$ as  the residue of $K[H]$, namely 
\begin{equation}
\label{eq:res-definition}
\res(H)=\dim_K K[H]/\tr(\omega_{K[H]})= | H\setminus \tr(H)|.
\end{equation}
Thus $\res(H)=0$ means that $H$ is symmetric, and $\res(H) \leq 1$ if and only if $H$ is nearly Gorenstein.

The conductor of the extension $K[H] \subseteq K[t]$ is the ideal
$$
\mathcal{C}_{K[t]/K[H]}=(t^h: h> \Fr(H))K[H],
$$ which explains why the quantity
$c(H):=\Fr(H)+1$ is named the conductor of $H$. We denote $\mathcal{C}_H=\{h: t^h\in \mathcal{C}_{K[t]/K[H]} \}$, which is an ideal in $H$ minimally generated by
$c(H), c(H)+1, \dots, c(H)+n_1-1$.
An important observation is that  $\tr(H)$ contains the conductor ideal $\mathcal{C}_H$. We skip the proof for now, since it follows from Proposition \ref{conductor} in the Appendix, where we consider a more general situation. 

\begin{Proposition}
\label{prop:subsets-trace} For any numerical semigroup $H$ one has
$$
\mathcal{C}_H \subseteq \tr(H) \subseteq H.
$$
If $H$ is not symmetric then $\mathcal{C}_H \subseteq \tr(H) \subseteq M$.
\end{Proposition}

As a corollary  we obtain an upper bound for $\res(H)$.

We define the set of {\em non-gaps} of $H$ to be
$NG(H)=\{ x\in H: x< \Fr(H)\}$ and we denote   $n(H)=|NG(H)|$.

\begin{Corollary}
\label{cor:boundres}
For any numerical semigroup $H$ one has $\res(H)\leq n(H)$, with equality if and only if $\tr(H)=\mathcal{C}_H$.
\end{Corollary}

\begin{proof}
The desired inequality follows from the observation that
$$
n(H)=|NG(H)|=|H\setminus \mathcal{C}_H| \geq |H\setminus \tr(H)|=\res(H).\quad \square
$$

\end{proof}

The map $\rho: NG(H)\to G(H)$ given by $\rho(x)=\Fr(H)-x$ for all $x$ in $NG(H)$  is well defined and injective.
Also, $|NG(H)|+ |G(H)|=\Fr(H)+1$, hence denoting $g(H)=|G(H)|$ we have  $n(H) \leq  g(H)$.

Numerical experiments with GAP (\cite{GAP}) indicate that   another bound for $\res(H)$ might also hold. 
We formulate the following question.

\begin{Question}
\label{que:g-n}
Given a numerical semigroup $H$, is it true that
$$
\res(H)\leq g(H)-n(H)?
$$
\end{Question}

This question has a positive answer for symmetric semigroups: by \cite[Lemma 1(f)]{FGH} $H$ is symmetric if and only if $n(H)=g(H)$.
In Proposition \ref{prop:3-bound} we also confirm Question \ref{que:g-n}, when $H$ is $3$-generated.
 
For any integer $a>3$ the semigroup $H=\langle a, a+1, \dots, 2a-1\rangle$ is nearly Gorenstein and not symmetric (see Proposition \ref{prop:arithmetic-ng}),
 and it has $\res(H)=1=n(H) < g(H)-n(H)=a-2$. 
This shows that the bound in Question \ref{que:g-n} is not always smaller than the one given by Corollary \ref{cor:boundres}.

\medskip

The second list of inclusions in Proposition~\ref{prop:subsets-trace} are sharp, as  confirmed by Proposition~\ref{prop:arithmetic-ng} and Example~\ref{ex:trace-conductor} below.

The following result shows that a numerical semigroup  generated by an arithmetic sequence is  nearly Gorenstein.
We also characterize when such semigroups are almost symmetric, taking into account that the symmetric case
was known from work of Gimenez, Sengupta and Srinivasan in  \cite{GSS}.

\begin{Proposition}
\label{prop:arithmetic-ng}
Let $e>2$, and $H=\langle a, a+d, \dots, a+(e-1)d\rangle$ with $a,d$ coprime nonnegative integers and $e\leq a$.
Then
\begin{enumerate}
\item[{\em (a)}] $H$ is nearly Gorenstein;
\item[{\em (b})] $H$ is symmetric if and only if $a \equiv 2 \mod (e-1)$;
\item[{\em (c)}] $H$ is almost symmetric  if and only if $a=e$ or $a\equiv 2 \mod (e-1)$.
\end{enumerate}
\end{Proposition}

\begin{proof}
It is known from \cite[Theorem 4.7]{GSS} that $\tau=\type(H)$ is the unique integer $1\leq \tau \leq e-1$ such that
$a=k(e-1)+\tau+1$ with $k$ integer. Equivalently, $k= \lfloor \frac{a-2}{e-1}  \rfloor$.

Tripathi \cite[Theorem on page 3]{Tripathi} shows  that
$$
\PF(H)=\left\{ a\left\lfloor \frac{x-1}{e-1} \right\rfloor +dx: a-\tau \leq x \leq a-1 \right\}.
$$
For $ a-\tau \leq x \leq a-1$ we get $k(e-1) \leq x-1 \leq k(e-1)+(\tau-1)$, hence  $\lfloor \frac{x-1}{e-1}\rfloor=k$.
This implies that $\Fr(H)=ak+d(a-1)$ and
\begin{equation}
\label{eq:pf-arithmetic}
\PF(H)=\{\Fr(H)-(\tau-1)d, \dots, \Fr(H)-d, \Fr(H)\},
\end{equation}
hence the canonical ideal $\Omega_H  $ is generated by
\begin{equation*}
\mathcal{W}=\{-\Fr(H), -\Fr(H)+d, \dots, -\Fr(H)+(\tau-1)d\}.
\end{equation*}

For part (a) we consider the set
\begin{equation*}
\mathcal{W'}=\{ \Fr(H)+a, \Fr(H)+a+d, \dots, \Fr(H)+a + (e-\tau)d \}\subset H.
\end{equation*}

An element in $\mathcal{W}+\mathcal{W}'$ is of the form $a+(i+j)d$ with $0\leq i\leq \tau-1$ and $1\leq j \leq e-\tau$.
This way we obtain the generators of $H$: $a, a+d, \dots,  a+(e-1)d$,
 which shows that $\mathcal{W}'\subset \Omega_H^{-1}$ and $\Omega_H+\Omega_H^{-1} \supseteq M$. Equivalently, $H$ is nearly Gorenstein.

Part (b) is known and may be traced  back to \cite[Theorem 2.2]{GSS} or (less explicitly in)  \cite{PatilSengupta}.
The statement is  an immediate consequence of the fact that $K[H]$ is Gorenstein if and only if $\tau=1$.

For part (c), using (b), it is enough to treat the case of $H$ being  almost symmetric, but not symmetric.
This is equivalent   (using \eqref{eq:ag-symmetries} and \eqref{eq:pf-arithmetic}) to
\begin{eqnarray*}
(\Fr(H)-(\tau-1)d) + (\Fr(H)-d) &=& \Fr(H), \text{ which is equivalent to} \\
 \Fr(H) &=& \tau d.
\end{eqnarray*}
After we substitute the values of $\Fr(H)$ and $\tau$ in the previous equation, we get
\begin{eqnarray*}
ak+ d(a-1) &=& (a-1-k(e-1)) d, \\
k(a+d(e-1)) &=& 0, \\
k &=& 0.
\end{eqnarray*}

Note that $e\leq a$ and by the way $k$ was defined, we may express $k= \left\lfloor \frac{a-2}{e-1}  \right\rfloor$.  Therefore $k=0$ if and only if $a=e$.
\end{proof}

Next, we present a family of numerical semigroups $H$ such that $ \mathcal{C}_H = \tr(H)$.

\begin{Example}
\label{ex:trace-conductor}
{\em
For the integers  $m>1$ and $q >0$ we let
$$
H=\langle m, qm+1, qm+2, \dots, qm+m-1 \rangle.
$$
This is a semigroup with minimal multiplicity, hence its pseudo-Frobenius numbers are obtained by subtracting $m$ from the rest of the minimal generators.  This gives $\PF(H)=\{(q-1)m+1, (q-1)m+2, \dots, qm-1\}$, a list of $m-1$ consecutive integers. Let $x\in \Omega_H^{-1}$, i.e. $-\PF(H)+x \subset H$. This can  happen
only if $x-\Fr(H)=x-qm+1 \geq qm$, equivalently $x\geq 2qm-1$. Consequently, $\tr(H)= \{x: x \geq qm \}=\mathcal{C}_H$, and $\res(H)=|\{0, m,\dots, (q-1)m\}|=q$.
}
\end{Example}

\section{The case of $3$-generated numerical semigroups}
\label{sec:3gens}

When the numerical semigroup $H$ is $3$-generated, the results in \cite[Section 3]{HHS}   can be applied to obtain a simple formula of $\res(H)$ from the defining ideal of $K[H]$.

Assume $H$ is minimally generated by $n_1, n_2, n_3$, not necessarily listed increasingly.  Let $\varphi: S=K[x_1, x_2, x_3] \to K[H]$ the algebra map given by
$\varphi(x_i)=t^{n_i}$ for $i=1, \dots, 3$. Then $\ker (\varphi)= I_H$, the defining ideal of $K[H]$.

It is proven in \cite{He-semi} that
$H$ is symmetric, equivalently  $K[H]$ is a complete intersection, if and only if,
 up to a permutation, $d=\gcd(n_1, n_2) >1$ and $n_3 \in \langle n_1/d, n_2/d \rangle$.

Assume $H$ is not symmetric. We recall from \cite{He-semi} how to compute the ideal $I_H$ in this case.
We find the positive integers $c_1, c_2, c_3$ minimal with the property that there exist nonnegative integers $a_i, b_i$, $i=1,\dots, 3$ such that
\begin{eqnarray}
\label{eq:3semi-equations}
\nonumber c_1 n_1 &=& b_2 n_2+a_3 n_3, \\
c_2 n_2 &=& a_1 n_1 + b_3 n_3, \\
\nonumber c_3 n_3 &=& b_1n_1 + a_2 n_2.
\end{eqnarray}

Such $a_i, b_i$ are positive, unique, and $c_i=a_i+b_i$ for $i=1, \dots, 3$.
In this notation, the ideal $I_H$ is the ideal of maximal minors of the matrix
\begin{eqnarray}
\label{structure}
A=\left( \begin{array}{ccc} x_1^{a_1} & x_2^{a_2} & x_3^{a_3}\\
x_2^{b_2} & x_3^{b_3} & x_1^{b_1}
 \end{array}\right),
\end{eqnarray}
that we call the {\em structure matrix} of the semigroup $H$.

It is noticed in \cite[page 69]{NNW}
that one can recover $n_1, n_2, n_3$ from  the matrix $A$
by computing  the $K$-vector space dimension for the isomorphic rings
$$
K[H]/(t^{n_1}) \cong S/(x_1, I_H) \cong K[ x_2, x_3]/(x_2^{a_2+b_2}, x_2^{b_2}x_3^{a_3}, x_3^{a_3+b_3}),
$$
and the other two cases, see \cite[Lemma 10.23]{RoSa-book} for a different approach.
Namely, we get
\begin{eqnarray}
\label{eq:gens-from-matrix}
\nonumber n_1 &=& a_2 a_3+b_2 a_3+ b_2 b_3, \\
n_2 &=& a_1a_3+a_1b_3 +b_1b_3, \\
\nonumber n_3 &=& a_1 a_2+b_1 a_2 +b_1b_2.
\end{eqnarray}

It follows from the Hilbert-Burch theorem (\cite[Theorem 1.4.17]{BH}) that the transpose $A^{T}$ is the relation matrix of $I_H$, i.e. the sequence
$$
0\to S^2  \stackrel{A^{T}}{\longrightarrow} S^3 \to I_H \to 0
$$
 is exact. The type of $R=K[H]$ is $2$, hence by \cite[Corollary 3.4]{HHS}  we get 
$$\tr(\omega_R)=I_1(\bar{A^T})=(t^{n_ia_i}, t^{n_ib_i}:i=1,\dots, 3),$$
where $\bar{A^{T}}$ is obtained by  applying $\varphi$ on the entries of $A^{T}$. We may formulate the following result.

\begin{Proposition}
\label{prop:3-semi-trace-formula}
Assume $H$ is a non-symmetric $3$-generated numerical semigroup and let $R=K[H]$.
With notation as in \eqref{eq:3semi-equations}, we set $d_i=\min\{a_i,b_i\}$ for $1\leq i \leq 3$.
Then
$$
\tr(\omega_R)= (t^{d_1 n_1},t^{d_2 n_2}, t^{d_3 n_3})R,  \text{  and } \res(H)=d_1d_2d_3.
$$
\end{Proposition}

\begin{proof} The first part is clear from the discussion above. Since
$$
R/\tr(\omega_R)  \iso S/(I_H, x_1^{d_1}, x_2^{d_2}, x_3^{d_3}) \iso S/(x_1^{d_1}, x_2^{d_2}, x_3^{d_3})
$$
we obtain that $\res(H)=\dim_K R/\tr(\omega_R)= d_1d_2d_3$.
\end{proof}

We may now give a positive answer to Question \ref{que:g-n}, in embedding dimension $3$.
\begin{Proposition}
\label{prop:3-bound}
For any $3$-generated numerical semigroup $H$ one has $$\res(H)\leq g(H)-n(H).$$
\end{Proposition}

\begin{proof}
If $H$ is symmetric we actually have equality $0=\res(H)= g(H)-n(H)$, as noted in \cite[Lemma 1(f)]{FGH}.
Assume $H$ is not symmetric and that it has a structure matrix $A$ denoted as in \eqref{structure}.
 Nari et al. prove in \cite[Theorem 3.2]{NNW} that
$$
2g(H)-(\Fr(H)+1) \in \{a_1 a_2 a_3, b_1 b_2 b_3\}.
$$
Using Proposition \ref{prop:3-semi-trace-formula}, we obtain
$$
\res(H) \leq \min\{a_1 a_2 a_3, b_1 b_2 b_3\} \leq 2g(H)-(\Fr(H)+1)=  g(H)-n(H). \quad \square
$$

\end{proof}

As an application of Proposition~\ref{prop:3-semi-trace-formula} we will  characterize the $3$-generated numerical semigroups such that their trace is at either end of the interval $[\mathcal{C}_H , M]$.

\begin{Theorem}
\label{thm:3semi-trace-is-maximal}
Let $H$ be a  $3$-generated numerical semigroup. Then $\tr(H)=M$ if and only if one of the following cases occurs:
\begin{enumerate}
\item[{\em (i)}] $H=\langle ab+b+1, b+c+1, ac+a+c\rangle$ where $a,b,c$ are positive integers with  $\gcd(b+c-1, ab-c)=1$, or
\item[{\em (ii)}] $H=\langle bc+b+1, ca+c+1, ab+a+1\rangle$, where $a,b,c$ are positive integers with $\gcd(bc+b+1, ca+c+1)=1$.
\end{enumerate}
In case {\em (i)}, $\Fr(H)= abc+bc-b-1+\max\{0, ab-c\}$, and in case {\em (ii)}, $\Fr(H)=2abc-2$.
\end{Theorem}

\begin{proof}
 Assume $H=\langle n_1, n_2, n_3 \rangle$ such that $\tr(H)=M$. By \cite[Corollary 3.5]{HHS}, that is equivalent to  $I_1(A)=(x_1, x_2, x_3)$,
where $A$ is the matrix attached to $H$ as in \eqref{structure}.
Clearly, $H$ is not symmetric, hence up to a permutation of the variables, there are essentially two (overlapping) cases to consider.

Case 1:
\begin{eqnarray*}
A=\left( \begin{array}{ccc} x_1^{ } & x_2^{a } & x_3^{b}\\
x_2^{ } & x_3^{ } & x_1^{c}
 \end{array}\right), \text{ with } a,b,c >0.
\end{eqnarray*}

Using \eqref{eq:gens-from-matrix} we get $n_1=ab+b+1, n_2= b+c+1, n_3=ac+a+c$, as desired.
It is easy to check that $\gcd(n_1, n_2)=\gcd(n_2, n_3)=\gcd(n_1, n_3)$, hence
$1=\gcd(n_1, n_2, n_3)= \gcd(n_2, n_1-n_2)= \gcd(b+c-1, ab-c)$.

Conversely, let $n_1=ab+b+1, n_2= b+c+1, n_3=ac+a+c$ for some positive integers $a,b,c$ such that $\gcd(b+c-1, ab-c)=1$.
Arguing as above we see that the generators of $H$ are pairwise coprime, hence  $H$ is a numerical semigroup which is not symmetric.
It is easy to verify the following equations:
\begin{eqnarray}
\label{eq:case1}
\nonumber (1+c) n_1 &=& n_2+ b n_3, \\
(1+a) n_2 &=& n_1 + n_3, \\
\nonumber (1+b) n_3 &=& c n_1 + a n_2.
\end{eqnarray}

We claim that these are the minimal relations \eqref{eq:3semi-equations} among $n_1, n_2, n_3$.

Since $a_1, b_3$ in \eqref{eq:3semi-equations} are positive, unique and $(1+a)n_2=n_1+n_3$, we may identify $c_2=1+a$ and $a_1=b_3=1$.

After substituting $n_1=(1+a)n_2-n_3$ into $c_1n_1=b_2 n_2+a_3n_3$, we get $c_1((1+a)n_2-n_3)=b_2n_2+a_3n_3$, hence
$$
(c_1(1+a)-b_2)n_2=(c_1+a_3) n_3.
$$
Since $n_2$ and $n_3$ are coprime, there exists a positive integer $\ell$ so that $c_1+a_3=\ell n_2$. Thus,  $c_1+a_3\geq b+c+1$.

On the other hand, comparing \eqref{eq:case1} and \eqref{eq:3semi-equations} we obtain that $c_1\leq 1+c $ and $a_3=c_3-b_3 \leq (b+1)-1=b$, hence
$c_1+a_3 \leq 1+c+b$. This implies that $c_1+a_3=b+c+1$, and moreover $c_1=c+1$ and $a_3=b$.
We can  now identify the rest of the coefficients in \eqref{eq:3semi-equations}:
$c_1=1+b, b_1=c, a_2=a$, which shows  that the matrix $A$ has the desired entries.

Case 2:
\begin{eqnarray*}
A=\left( \begin{array}{ccc} x_1^{ } & x_2^{ } & x_3^{ }\\
x_2^{b} & x_3^{c} & x_1^{a}
 \end{array}\right), \text{ with } a,b,c >0.
\end{eqnarray*}

Using \eqref{eq:gens-from-matrix} we get $n_1=bc+b+1, n_2=ca+c+1, n_3=ab+a+1$.
It is easy to see that   $\gcd(n_1, n_2)=\gcd(n_2, n_3)=\gcd(n_1, n_3)$, hence the desired  description for $H$.

Conversely, let  $a,b,c$  be  positive integers with $\gcd(bc+b+1, ca+c+1)=1$.
It now follows from \cite[Theorem 14]{RoSa-3pseudo} (and its proof) that    $H=\langle bc+b+1, ca+c+1,  ab+a+1   \rangle$ is a pseudo-symmetric numerical semigroup
whose matrix $A$  is the one we started with this case.

It is shown in \cite[Theorem 2.2.3]{RamirezAlfonsin} and \cite[Exercise 5, pp. 145]{Kunz-book} that
for any non-symmetric numerical semigroup $H=\langle n_1,n_2, n_3 \rangle$ one has
$$
\Fr(H)= \max\{ c_1n_1+b_3n_3, c_2n_2+ a_3n_3\},
$$
where $c_1, c_2, a_3, b_3$ are as in \eqref{eq:3semi-equations}.
It is now an easy exercise to derive  the announced formulas for $\Fr(H)$, when $H$ belongs to either one of the two families considered above.
\end{proof}

\begin{Remark}
{\em
As noticed by  Nari, Numata and Keiichi~Watanabe in \cite[Corollary 3.3]{NNW} (see also \cite[Corollary 2.9]{Numata}), the format of the matrix $A$ in case (ii) of Theorem \ref{thm:3semi-trace-is-maximal}
corresponds to $H$ being pseudo-symmetric, which is equivalent in  embedding dimension $3$ to $H$ being almost symmetric and not symmetric, see \cite[Proposition 2.3]{Numata}.
The complete  parametrization of $3$-generated pseudo-symmetric numerical semigroups was obtained by Rosales and Garc\'ia-S\'anchez in \cite{RoSa-3pseudo}.
}
\end{Remark}

\begin{Proposition}
\label{prop:3semi-trace-is-conductor}
Assume $H$ is a non-symmetric $3$-generated numerical semigroup. Then $\tr(H)=\mathcal{C}_H$ if and only if $H=\langle 3, 3a+1, 3a+2\rangle$ for some positive integer $a$.
\end{Proposition}

\begin{proof}
Assume $\tr(H)=\mathcal{C}_H$.

It follows from Proposition \ref{prop:3-semi-trace-formula} that 
 $\mu(\tr(H))\leq 3$. If $\mu(\tr(H))=2$, then $e(H)=\mu(\mathcal{C}_H)=2$ and $H$ is symmetric, a contradiction. Therefore, 
$\mu(\tr(H))=3$,  hence  $H$ has multiplicity $3$.
(We can get the same thing by applying   Corollary \ref{maxtr}.)
Listing  its generators increasingly we have that either $H=\langle 3, 3a+1, 3b+2 \rangle$ with $0<a \leq b$, or $H=\langle 3, 3b+2, 3a+1\rangle$ with $a>b>0$.

Assume $a \leq b$. Then $3b+2\notin \langle 3, 3a+1\rangle$, hence $3b+2 \leq \Fr(\langle 3, 3a+1\rangle)= 6a-1$, by \cite{RamirezAlfonsin}. Thus $b<2a$.
It is easy to check that the structure matrix   \eqref{structure} is
$$
A=\left( \begin{array}{ccc} x_1^{2a-b} & x_2 & x_3\\
x_2 & x_3 & x_1^{2b-a+1}
 \end{array}\right).
$$
Hence  $\res(H)=2a-b$  by   Proposition \ref{prop:3-semi-trace-formula}.
Note that $0, 3, 6, \dots, 3(a-1)$ are not in $\mathcal{C}_H$, hence
$2a-b=\res(H)=|H\setminus{C}_H| \geq  a$. This gives $a=b$ and $H= \langle 3, 3a+1, 3a+2 \rangle$.

If $a > b$ then  arguing as in the previous case we obtain $3a+1 \leq \Fr(\langle 3, 3b+2 \rangle)= 6b+1$, and $a<2b$.
Clearly $0, 3, 6, \dots, 3b$ are  not in $\mathcal{C}_H$, hence $\res(H)=|H\setminus \mathcal{C}_H| \geq b+1$.
On the other hand, the structure matrix of $H$ is
$$
A=\left( \begin{array}{ccc} x_1^{2b-a+1} & x_2 & x_3\\
x_2  & x_3 & x_1^{2a-b}
 \end{array}\right),
$$
 and Proposition \ref{prop:3-semi-trace-formula} gives $\res(H)=2b-a+1$.
Thus $2b-a+1 \geq b+1$, and $b\geq a$, a contradiction.

Example \ref{ex:trace-conductor}  confirms that for any $a>0$ the semigroup $H=\langle 3,3a+1,3a+2 \rangle$ satisfies $\tr(H)=\mathcal{C}_H$.
\end{proof}

\begin{Remark} \label{rem:3large}
{\em
From the proof of Proposition \ref{prop:3semi-trace-is-conductor} we see that for any $a>0$ we have that $\res(\langle 3, 3a+1, 3a+2\rangle)= a$.
}
\end{Remark}

\section{The residue in shifted families of semigroups}
\label{sec:shifted-residue}

Remark~\ref{rem:3large} indicates that  the residue of a $3$-generated numerical semigroup $H$ may be as large as possible.
However, this is not the case in a shifted family of semigroups, as we verify below.

Firstly, we extend the definition of residue from \eqref{eq:res-definition} to arbitrary affine subsemigroups of $\NN$.
In this sense, for any semigroup $H\subset \NN$ containing $0$ we let
\begin{equation*}
\res(H) = \res \left(\frac{1}{d}H\right), \text{ where } d=\gcd(h:h\in H).
\end{equation*}

Given the sequence of integers $\ab: a_1<\dots <a_e$, for any $j$ we denote $\ab+j:a_1+j, \dots, a_e+j$.
The shifted family of $\ab$ is the family $\{ \ab+j \}_{j\geq 0}$.
It has been proved that for large enough shifts several properties occur periodically in the shifted
family of semigroups $\{ \langle \ab+j \rangle \}_{j\geq 0}$
and their semigroup rings  $\{ K[\langle \ab+j \rangle ]\}_{j\geq 0}$, see \cite{JS}, \cite{Vu}, \cite{HeS}, \cite{S-3semi}.
For instance, Jayanthan and Srinivasan  \cite{JS}  showed that  $\text{for }j\gg 0 $
$$
  K[\langle \ab+j \rangle] \text{ is complete intersection (CI)} \iff K[\langle \ab+j +(a_e-a_1) \rangle] \text{ is CI}.
$$
More generally, Vu (\cite[Theorem 1.1]{Vu}) showed that
\begin{equation*}
\text{for }j\gg 0, \quad \beta_i(K[\langle \ab+j \rangle])=\beta_i(K[\langle \ab+j+(a_e-a_1) \rangle]) \text{ for all }i.
\end{equation*}
In particular, for $j\gg 0$ the algebras $K[\langle\ab+j\rangle]$ and $K[\langle \ab+j+(a_e-a_1)\rangle]$ are Gorenstein at the same time. This implies that
 the semigroups $\langle\ab+j\rangle$ and $\langle \ab+j+(a_e-a_1)\rangle$ are symmetric at the same time. Equivalently,
\begin{equation*}
\text{ for } j \gg 0, \quad \res(\langle \ab+j \rangle)=0 \iff \res(\langle \ab+j+ (a_e-a_1) \rangle) =0.
\end{equation*}

It is natural to ask the following.

\begin{Question}
\label{que:shifts}
Given the list of integers $\ab: a_1<\dots <a_e$, is it true that
\begin{equation*}
\text{ for } j\gg 0, \quad \res(\langle \ab+j \rangle) = \res(\langle \ab+j+ (a_e-a_1) \rangle) ?
\end{equation*}
\end{Question}

We remark that a positive answer to Question~\ref{que:shifts} implies that for numerical semigroups with bounded width, their residue is also bounded. We recall that the width of a numerical semigroup  $H$  is defined in \cite{HeS} as the difference between the largest and the smallest minimal generator of $H$.

Numerical experiments with GAP (\cite{GAP}, \cite{Num-semigroup}) indicate that Question \ref{que:shifts} might have a positive answer.
Next we confirm it in case $e\leq 3$. If $e=2$, then $\langle a_1+j, a_2+j\rangle$ is symmetric for all $j$, and we are done.
The case $e=3$ is proved in the following theorem.
We first note that when studying  asymptotic properties in  a shifted  family $\{ \ab +j \}_j$, we may assume $a_1=0$.

\begin{Theorem}
\label{thm:3semi-res-periodic}
Given the integers $0<a<b$, let $D=\gcd(a,b)$ and  $k_{a,b}= \max\{b(\frac{b-a}{D} -1), \frac{ba}{D} \}$.
For any  integer $j$ we denote $H_j= \langle j, j+a, j+b\rangle$.

Then $\res(H_j)=\res(H_{j+b})$ for all $j>2k_{a,b}$.
\end{Theorem}

Before giving the proof of Theorem \ref{thm:3semi-res-periodic} we recall a result from \cite{S-3semi} (extending  Jayanthan and Srinivasan's \cite[Theorem 1.4]{JS})
 about the occurence of symmetric semigroups in a shifted family $\{ \langle j, a+j, b+j \rangle \}_{j\geq 0}$.

\begin{Lemma} (\cite[Theorem 3.1]{S-3semi})
\label{lemma:3semi-ci}
With notation as in Theorem \ref{thm:3semi-res-periodic}, let
\begin{equation}
\label{eq:T}
T=\prod_{p \text{ prime, } \nu_p(a)<\nu_p(b)} p^{\nu_p(b)},
\end{equation}
where for any integer $n$ we denote $\nu_p(n)=\max\{i: p^i \text{ divides }n\}$.
Then for $j>k_{a,b}$  the semigroup $H_j$ is symmetric if and only if $j$ is a multiple of $T$.
In particular, in the family of semigroups $\{ H_j \}_{j> k_{a,b}}$ the symmetric property occurs periodically with principal period $T$.
\end{Lemma}

\begin{proof} (of Theorem \ref{thm:3semi-res-periodic}).

We start with $j>k_{a,b}$.
By Lemma \ref{lemma:3semi-ci}, if $H_j$ is symmetric  then $H_{j+b}$ is symmetric, too, hence $\res(H_j)=\res(H_{j+b})=0$.

Assume  $H_j$ is not symmetric. By Lemma \ref{lemma:3semi-ci}, $H_{j+\ell b}$ is not symmetric for all $\ell \geq 0$.
Denote $A_{j+\ell b}$   the structure matrix \eqref{structure} of the non-symmetric semigroup $H_{j+\ell b}$.

For $(n_1, n_2, n_3)=(0, a, b)+j+\ell b$, it is proved in \cite[Theorem 2.2]{S-3semi} that for any $\ell \geq 0$
the middle equation in \eqref{eq:3semi-equations} is
\begin{equation}
\label{eq:middle}
\frac{b}{D} n_2= \frac{b-a}{D}n_1+ \frac{a}{D}n_3.
\end{equation}

This implies that
\begin{eqnarray*}
A_j=\left( \begin{array}{ccc} x_1^{(b-a)/D} & x_2^{a_2} & x_3^{a_3}\\
x_2^{b_2} & x_3^{a/D} & x_1^{b_1}
 \end{array}\right),
\end{eqnarray*}
where $a_2, a_3, b_1, b_2$  are positive integers (depending on $j$) such that $a_2+b_2=b/D$, by \eqref{eq:middle}.

Let $e=\gcd(a,b)/\gcd(j,a,b)$. Proposition 4.2 in \cite{S-3semi} explains how the equations \eqref{eq:3semi-equations} change when we shift up by $b$.
According to this result, only the last column of $A_j$ changes and  we obtain
\begin{eqnarray*}
A_{j+b}=\left( \begin{array}{ccc} x_1^{(b-a)/D} & x_2^{a_2} & x_3^{a_3+e}\\
x_2^{b_2} & x_3^{a/D} & x_1^{b_1+e}
 \end{array}\right).
\end{eqnarray*}
Iterating this, we have that
\begin{eqnarray*}
A_{j+\ell b}=\left( \begin{array}{ccc} x_1^{(b-a)/D} & x_2^{a_2} & x_3^{a_3+\ell e}\\
x_2^{b_2} & x_3^{a/D} & x_1^{b_1+\ell e}
 \end{array}\right), \text{ for }\ell \geq 0.
\end{eqnarray*}

Proposition \ref{prop:3-semi-trace-formula} gives
\begin{equation}
\label{eq:after-shift-2}
\res(H_{j+\ell b})=\min \{(b-a)/D, b_1+\ell e \} \cdot \min \{ a/D, a_3+\ell e\} \cdot \min \{a_2, b_2\}.
\end{equation}

For $\ell \geq \max\{ \frac{b-a}{D}-1, \frac{a}{D} \} = \frac{1}{b} k_{a,b}$ it is easy to see that $b_1+\ell e \geq (b-a)/D$ and
$a_3+\ell e \geq a/D$. Hence, \eqref{eq:after-shift-2} becomes
\begin{equation}
\label{eq:res-big-shift}
\res(H_{j+\ell b})= \min\{a_2, b_2\} \cdot a(b-a)/D^2,
\end{equation}
which is a formula not involving $\ell$.

The argument above shows that for any $j>2k_{a,b}$ we have that $\res (H_j)=\res(H_{j+b})$. This concludes the proof.
\end{proof}

\begin{Corollary}
\label{cor:bound-res-shifts}
With notation as in Theorem \ref{thm:3semi-res-periodic}, for $j>2k_{a,b}$ the residue of $H_j$ is an integer divisible by $(b-a)a/D^2$ and
$$
\res(H_j) < 8b^3/27D^3.
$$
\end{Corollary}
	
\begin{proof}
If $H_j$ is symmetric, the inequality to prove is clear.
Assume $H_j$ is not symmetric and $j>2k_{a,b}$. By \eqref{eq:res-big-shift}, we have $\res(H_{j })= \min\{a_2, b_2\} \cdot a(b-a)/D^2$, with
$a_2, b_2$ positive integers such that $a_2+b_2=b/D$.  This shows the first part of the claim. The second part is obtained from the following chain of inequalities
$$
\res(H_j) \leq \frac{a(b-a)}{D^2} \left(\frac{b}{D}-1\right) < \frac{ab(b-a)}{D^3} \leq \left( \frac{2b}{3} \right)^3 \cdot \frac{1}{D^3}=\frac{8b^3}{27D^3},
$$
where for the last inequality we used the known fact that  $\sqrt[3]{xyz} \leq (x+y+z)/3$ for $x,y,z >0$.
\end{proof}

\begin{Corollary}
With notation as in Theorem \ref{thm:3semi-res-periodic}, for $j>2k_{a,b}$ the semigroup $H_j$ is nearly Gorenstein if and only if $H_{j+b}$ is nearly Gorenstein.
\end{Corollary}

We make a comment about the frequency of occurences of the symmetric, almost symmetric and nearly Gorenstein property in a shifted family.

\begin{Remark}
{\em Let $0<a<b$. For $j\geq 0$ we denote $H_j=\langle j, j+a, j+b\rangle$. We use the constant $k_{a,b}$ introduced in Theorem \ref{thm:3semi-res-periodic}.

 Lemma \ref{lemma:3semi-ci} shows that we find symmetric semigroups for arbitrarily large shifts $j$.
From the formula \eqref{eq:T} for the principal period $T$ we infer that $T >1$, otherwise $a=b$, which is false.
This means that there is no $j_0$ such that $H_j$ is symmetric for all $j>j_0$.

When $b=2a$, the semigroup $H_j$ is generated by an arithmetic sequence.
Using  Proposition \ref{prop:arithmetic-ng}  we get that $H_j$ is nearly Gorenstein for all $j>0$.
On the other hand, by Lemma \ref{lemma:3semi-ci}, when $j>b$ we have that $H_j$ is symmetric if and only if  $j$ is divisible by $2^{\nu_2(b)}$.

It is however possible that in the shifted family $\{H_j\}_{j\geq 0}$, for large 	$j$ the only nearly Gorenstein semigroups are the symmetric ones.
Indeed, if $a, b$ are coprime and $a>1$, by Corollary \ref{cor:bound-res-shifts} we have that for $j>2k_{a,b}$,  either $\res(H)=0$, or $\res(H) \geq a(b-a) >1$.

The first author  was informed by  Kei-ichi Watanabe  that there are only finitely many almost symmetric semigroups in the shifted family $\{H_j\}_{j\geq 0}$.
This can also be seen as follows.
  According to Nari et al. \cite{NNW}, the structure matrix $A_j$ for an almost symmetric semigroup $H_j$ must have one row consisting
of linear forms. However, it is proven in \cite[Proposition 4.2]{S-3semi} that for $j> k_{a,b}$ the matrix $A_{j+b}$ is obtained from $A_j$
by increasing the exponents of the last column  by $\gcd(a,b)/\gcd(a,b,j)$. This shows that for $j>k_{a,b}+b$ the semigroup $H_j$ is not almost symmetric. The analysis of the occurence of infinitely many almost symmetric semigroups  in the shifted family of a $4$-generated numerical semigroup is made in  \cite[Section 7]{HeWa}.
}
\end{Remark}

\appendix 
\section{Trace ideals as conductor ideals in local rings}

Some statements in this paper can be formulated for extensions of $1$-dimensional local domains where the conductor is defined. In this appendix we would like to understand the case when the trace of an ideal coincides  with the conductor ideal, and some consequences this draws.

Hereafter, $(R,\mm)$ is a $1$-dimensional Cohen-Macaulay local ring with canonical module $\omega_R$.
Assume now that $(R,\mm)\subset (\widetilde{R},\nn)$ is an extension of local rings,  where $\widetilde{R}$ is a finite $R$-module and a discrete valuation ring
such that $\widetilde{R}\subseteq Q(R)$.  We also assume  that the inclusion map $R\to \widetilde{R}$ induces an isomorphism $R/\mm\to \widetilde{R}/\nn$. 
The set 
$$
\mathcal{C}_{\widetilde{R}/R}= \{x \in R: x \widetilde{R} \subseteq R \}
$$
is called the {\em conductor} of this extension and it is an ideal of both $R$ and $\widetilde{R}$.  
With the notation introduced we  have

\begin{Proposition}
\label{conductor}
For any  ideal $I\subset R$ one has $\mathcal{C}_{\widetilde{R}/R}\subseteq \tr(I)$.
In particular, $\mathcal{C}_{\widetilde{R}/R}\subseteq\tr(\omega_R)$.
\end{Proposition}

\begin{proof}
There exists $f\in I$ such that $I\widetilde{R}=(f)\widetilde{R}$.
Now let $g\in \mathcal{C}_{\widetilde{R}/R}$. Then $(g/f)I\subseteq (g/f)I\widetilde{R}=g\widetilde{R}\subseteq R$.
Thus $g/f\in I^{-1}$. Since $g=(g/f)f$ with $f\in I$, it follows that $g\in I^{-1}\cdot I=\tr(I)$, where for the last equation we used \cite[Lemma 1.1]{HHS}.
\end{proof}

\begin{Corollary}
\label{mmtrace}
$R$ is nearly Gorenstein, if $\mathcal{C}_{\widetilde{R}/R}=\mm$.
\end{Corollary}

For the rest of this section, the ideal quotients are computed in $Q(R)=Q(\widetilde{R})$.  The following lemma will be used in the proof of Proposition \ref{end}.

\begin{Lemma}
\label{lemma:rtilde}
$ \widetilde{R}=R:\mathcal{C}_{\widetilde{R}/R} $
\end{Lemma}

\begin{proof}
It is clear from the definition of $\mathcal{C}_{\widetilde{R}/R}$ that $ \widetilde{R} \subseteq R:\mathcal{C}_{\widetilde{R}/R}$.

For the reverse inclusion, let $f\in  R:\mathcal{C}_{\widetilde{R}/R}$. If $f\notin \widetilde{R}$, then since $Q(R)=Q(\widetilde{R})$ 
we may write $f=\epsilon t^{-a}$ with $\epsilon$ invertible in $\widetilde{R}$, $a$ positive integer and $t$   a generator of  the maximal ideal of $\widetilde{R}$.
Since $\mathcal{C}_{\widetilde{R}/R}$ is an ideal in $\widetilde{R}$, there exists $b>0$ such that $\mathcal{C}_{\widetilde{R}/R}=(t^b)\widetilde{R}$.

The property $f\in R:\mathcal{C}_{\widetilde{R}/R}$ now reads as $\epsilon t^{-a} \mathcal{C}_{\widetilde{R}/R} \subseteq R$, or equivalently, that
$ \mathcal{C}_{\widetilde{R}/R} \subseteq t^a R$. This implies that $b\geq a$.
Since we may write $t^{b-1}= f\cdot g$ where $g=\epsilon^{-1}t^{a+b-1}$ is clearly in $\mathcal{C}_{\widetilde{R}/R}$, it follows that $t^{b-1}\in R$.
 
We claim that $t^{b-1}\widetilde{R} \subseteq  R$. This will then lead to a contradiction, since $t^{b-1}\not\in \mathcal{C}_{\widetilde{R}/R}$.
It is clear that $t^{b-1}\nn= \mathcal{C}_{\widetilde{R}/R}\subseteq R$. 
Thus is suffices to show that if $\nu\in\widetilde{R}$ is a unit, then $\nu t^{b-1}\in R$.
To prove this, we use our assumption that $R/\mm\to \widetilde{R}/\nn$ is an isomorphism.
In order to simplify notation we may assume that $R/\mm= \widetilde{R}/\nn$.
Then this  implies that  there exists $h\in R$ such that $\nu-h\in \nn$.
Thus, $\nu =h+h_1$ with $h\in R$ and $h_1\in \nn$, and therefore  $\nu t^{b-1}=h t^{b-1}+h_1 t^{b-1}$.
Since both summands on the right hand side of this equation belong to $R$, the claim follows.

This concludes the proof of the inclusion $R:\mathcal{C}_{\widetilde{R}/R} \subseteq \widetilde{R}$.
\end{proof}

 Our next result gives several characterizations of the situation when $\tr(I)=\mathcal{C}_{\widetilde{R}/R}$, in terms of the fractionary ideal $I^{-1}$.
We first recall that  for any fractionary ideal $J$ of $R$,  the  ideal quotient $J:J$ may be identified with the endomorphism ring $\End(J)$ of $J$, see \cite[Lemma 3.14]{Lindo}.

\begin{Proposition}
\label{end}
Let $I\subset R$ be an ideal. The following conditions are equivalent:
\begin{enumerate}
\item[{\em (i)}] $\tr(I)=\mathcal{C}_{\widetilde{R}/R}$.

\item[{\em (ii)}] $\End(I^{-1})=\widetilde{R}$.

\item[{\em (iii)}] $I^{-1}\iso \widetilde{R}$.
\end{enumerate}
\end{Proposition}

\begin{proof}
(i) \implies(ii): If $\tr(I)=\mathcal{C}_{\widetilde{R}/R}$, then using Lemma \ref{lemma:rtilde} and the  remark following it, we get
$$
 \widetilde{R}=R:\mathcal{C}_{\widetilde{R}/R}   =R:(I \cdot I^{-1})=(R:I):I^{-1}=\End(I^{-1}).
$$

(ii)\implies (i): Suppose that $\tr(I)\neq \mathcal{C}_{\widetilde{R}/R}$.  Then $\mathcal{C}_{\widetilde{R}/R}$ is properly contained in $\tr(I)$.
It follows that
 $\mathcal{C}_{\widetilde{R}/R}$ is properly  contained in $\tr(I)\widetilde{R}$. Let $t$ be  generator of the maximal ideal of $\widetilde{R}$.
Then there exist integers $a <b$ such that $\tr(I)\widetilde{R}=t^a\widetilde{R}$ and $\mathcal{C}_{\widetilde{R}/R}=t^b\widetilde{R}$.
  Clearly, $t^{b}\in \mathcal{C}_{\widetilde{R}/R}\widetilde{R}=\mathcal{C}_{\widetilde{R}/R}$.
Since $\tr(I)\widetilde{R}$ is a principal ideal, there exists $f\in \tr(I)$ such that $\tr(I)\widetilde{R}= (f)\widetilde{R}$. 
Therefore, there exists $u$ invertible in $\widetilde{R}$ such that $t^a=u\cdot f$.

Since by assumption  $\widetilde{R}=\End(I^{-1})$ (which is $R:\tr(I)$),   we obtain that $t^a\in R$. We may write $t^{b-1}=t^{b-a-1}\cdot t^a$, 
where  $t^{b-a-1}\in\widetilde{R}$ and $t^a\in \tr(I)$. Using again that $\widetilde{R}=R:\tr(I)$, it follows that $t^{b-1}\in R$.
Arguing as in the proof of Lemma \ref{lemma:rtilde} we obtain $t^{b-1} \in \mathcal{C}_{\widetilde{R}/R}$, which is a contradiction.

(ii)\implies  (iii):   If $\End(I^{-1})=\widetilde{R}$, then $I^{-1}\widetilde{R}=I^{-1}$. Since any nonzero $\widetilde{R}$-ideal is isomorphic to $\widetilde{R}$, the assertion follows.

(iii)\implies (ii) is obvious.
 \end{proof}

For an $R$-module $M$ we let $e(M)$ denote its multiplicity.

\begin{Corollary}
\label{maxtr}
Let $R$ be as before,  and assume in addition that   $R$ is of the form
$R=S/J$  with $(S,\nn)$ regular local ring of dimension $3$ and $J\subseteq \nn^2$.

If $\tr(\omega_R)= \mathcal{C}_{\widetilde{R}/R}$ then   $e(R) \leq \mu(J)$.
Moreover, if $R$ is an almost complete intersection, then $R$ has minimal multiplicity.
\end{Corollary}

\begin{proof}
Since $R$ is a $1$-dimensional domain, it it Cohen-Macaulay and $\projdim_S (R)=1$. 
Thus $J$ has a minimal presentation
       $ 0 \to  S^{g-1} \to S^{g} \to J$, where $g=\mu(J)$. 

Now Proposition~\ref{end}(iii) together with   $\tr(\omega_R)= \mathcal{C}_{\widetilde{R}/R}$ imply that $\omega_R^{-1}$ and $\widetilde{R}$ are
isomorphic $R$-modules, and they must have the same number of minimal generators over $R$.
Hence $\dim_{R/\mm}(\widetilde{R}/\mm \widetilde{R})=\mu(\omega_R^{-1})  \leq g$, by  \cite[Corollary 3.4]{HHS}.

As $\widetilde{R}$ is a discrete valuation ring, there exists $f\in \mm$ such that $\mm \widetilde{R}=f\widetilde{R}$.
Since $\widetilde{R}$ is a finitely generated $R$-module of rank $1$, it follows  that $e(R)=e(\widetilde{R})$,
see \cite[Corollary 4.7.9]{BH}.
Now $e(\widetilde{R})=\dim_{R/\mm}\mm^k\widetilde{R}/\mm^{k+1}\widetilde{R}$ for $k\gg 0$.
Since $\mm^k\widetilde{R}/\mm^{k+1}\widetilde{R}=f^k\widetilde{R}/f^{k+1}\widetilde{R}\iso \widetilde{R}/f\widetilde{R} \iso \widetilde{R}/\mm\widetilde{R}$,
we conclude from the above considerations that  $e(R)\leq g$, as desired.

If $R$ is an almost complete intersection, then $\mu(J)=\height (J)+1= \dim  S$, hence  $e(R) \leq 3$. 
On the other hand, a celebrated inequality of Abhyankar gives  $3=\embdim R \leq e(R)+\dim(R)-1=e(R)$. Thus $e(R)=3$, as desired.
\end{proof}

\medskip
{\bf Acknowledgement}.
We gratefully acknowledge the use of  SINGULAR (\cite{Sing}) and of the numericalsgps package (\cite{Num-semigroup})  in GAP (\cite{GAP}) for our computations.

Dumitru Stamate  was partly supported by a fellowship at the Research Institute of the University of Bucharest (ICUB) and by the University of Bucharest, Faculty of Mathematics and Computer Science through the 2019 Mobility Fund.
{}

\medskip


\begin{thebibliography}{}
 

\bibitem{BF} V.~Barucci,  R.~Fr\"oberg, \textit{One-dimensional almost Gorenstein rings}, J. Algebra  {\bf 188} (1997), 418--442.
 
\bibitem{brennen-et-al} J.P.~Brennan, L.~Ghezzi, J.~Hong, L.~Hutson, W.V.~Vasconcelos, \textit{Canonical Degrees of Cohen-Macaulay Rings and Modules: a Survey},  Preprint 2020, arXiv:2006.14401 [math.AC].

\bibitem{BH} W.~Bruns, J.~Herzog, \textit{Cohen-Macaulay Rings}, Revised Ed., Cambridge Stud. Adv. Math., vol. {\bf 39}, Cambridge University Press, Cambridge, 1998.
 
\bibitem{C-all} R.~Conaway, F.~Gotti, J.~Horton, R.~Pelayo, M.~Williams, B.~Wissman, 
\textit{Minimal presentations of shifted numerical monoids}, 
International Journal of Algebra and Computation {\bf 28} (2018),  53--68.
 

\bibitem{Sing} W.~Decker, G.-M.~Greuel, G.~Pfister, H.~Sch{\"o}nemann,
\newblock {\sc Singular} {4-1-2} --- {A} computer algebra system for polynomial computations.
\newblock {http://www.singular.uni-kl.de} (2019).

\bibitem{Num-semigroup}  M.~Delgado, P.~A.~Garc\'ia-S\'anchez, J.~Morais, \textit{numericalsgps}: a GAP package on numerical
semigroups. (http://www.gap-system.org/Packages/numericalsgps.html).
 

\bibitem{Eis} D.~Eisenbud, \textit{Commutative Algebra with a View Toward Algebraic Geometry}, Graduate Texts in Mathematics {\bf 150}, Springer, 1995.

\bibitem{FGH} R.~Fr\"oberg, C.~Gottlieb, R.~H\"aggkvist, \textit{On numerical semigroups}, Semigroup Forum {\bf 35} (1987), 63--83.

\bibitem{GAP}  The GAP Group. GAP -- Groups, Algorithms, and Programming, Version 4.7.5, 2014. (http://www.gap-system.org).

\bibitem{GSS} P.~Gimenez, I.~Sengupta, H.~Srinivasan, \textit{Minimal graded free resolutions for monomial curves defined by arithmetic sequences}, J. Algebra {\bf 338} (2013), 294--310.

\bibitem{GMT} S.~Goto, N.~Matsuoka, T.~Thi Phuong, \textit{Almost Gorenstein rings}, J. Algebra {\bf 379} (2013), 355--381.
  
\bibitem{He-semi} J.~Herzog, \textit{Generators and relations of Abelian semigroups and semigroup rings}, Manuscripta Math. {\bf 3} (1970), 175--193. 

\bibitem{HHS} J.~Herzog, T.~Hibi, D.I.~Stamate, \textit{The trace of the canonical module}, Israel Journal of Mathematics {\bf 233} (2019), 133--165.

\bibitem{HeS} J.~Herzog, D.I.~Stamate, \textit{On the defining equations of  the tangent cone of a numerical semigroup ring}, J. Algebra {\bf 418} (2014), 8--28.

\bibitem{HeWa} J.~Herzog, K.-i~Watanabe, \textit{Almost symmetric numerical semigroups}, Semigroup Forum {\bf 98} (2019), 589--630. 

\bibitem{JS} A.V.~Jayanthan, H.~Srinivasan, \textit{Periodic occurence of complete intersection monomial curves}, Proc. Amer. Math. Soc. {\bf 141} (2013), 4199--4208.

\bibitem{Kunz} E.~Kunz, \textit{The value-semigroup of a one-dimensional Gorenstein ring}, Proc. Amer. Math. Soc. {\bf 25} (1970),  748--751.

\bibitem{Kunz-book} E.~Kunz, \textit{Einf\"uhrung in die kommutative Algebra und algebraische Geometrie}, Vieweg, Braunschweig/Wiesbaden, 1980.

\bibitem{Lindo} H.~Lindo, \textit{Trace ideals and centers of endomorphism rings of modules over commutative rings}, J. Algebra {\bf 482} (2017), 102--130.
  
\bibitem{MoSt} A.~Moscariello, F.~Strazzanti, \textit{Nearly Gorenstein versus almost Gorenstein affine monomial curves}. Preprint 2020, 	arXiv:2003.05391 [math.AC].

\bibitem{Nari} H.~Nari, \textit{Symmetries on almost symmetric numerical semigroups}, Semigroup Forum {\bf 86} (2013),  140--154.

\bibitem{NNW} T.~Nari, T.~Numata, K.-i~Watanabe, \textit{Genus of numerical semigroups generated by three elements}, J. Algebra {\bf 358} (2012), 67--73.

\bibitem{Numata} T.~Numata, \textit{Almost symmetric numerical semigroups with small number of generators}, Ph.D.~Thesis, Nihon University, Tokyo, 2015.

\bibitem{OP} C.~O'Neill, R.~Pelayo, \textit{Ap\`ery sets of shifted numerical monoids}, Advances in Applied Mathematics {\bf 97} (2018), 27--35.

\bibitem{PatilSengupta} D.P.~Patil, I.~Sengupta, \textit{Minimal set of generators for the derivation module of certain monomial curves}, Comm. Algebra {\bf 27} (1999), 5619--5631.

\bibitem{RamirezAlfonsin} J.L.~Ram\'irez~Alfons\'in, \textit{The Diophantine Frobenius problem}, Oxford Lecture Series in Mathematics and its Applications, vol. {\bf 30}, Oxford University Press, 2005.

\bibitem{RoSa-3pseudo} J.C.~Rosales, P.A.~Garc\'ia-S\'anchez, \textit{Pseudo-symmetric numerical semigroups with three generators},  J. Algebra {\bf 291} (2005), 46--54.

\bibitem{RoSa-book} J.C.~Rosales, P.A.~Garc\'ia-S\'anchez, \textit{Numerical Semigroups},  Developments in Mathematics, vol. {\bf 20}, Springer,  2009.

\bibitem{S-3semi} D.I.~Stamate, \textit{Asymptotic properties in the shifted family of a numerical semigrou with few generators}, Semigroup Forum {\bf 93} (2016), 225--246.

\bibitem{St-surveybetti} D.I.~Stamate, \textit{Betti numbers for numerical semigroup rings}, In: Multigraded Algebra and Applications (V. Ene, E. Miller, Eds.), Springer Proceedings in Mathematics \& Statistics, vol. 238 (2018), 133--157, Springer, Cham.

 \bibitem{Tripathi} A.~Tripathi, \textit{On a variation of the coin exchange problem for arithmetic progressions},  Integers {\bf 3} (2003), A1, 5 pp.
 
\bibitem{Vu} T.~Vu, \textit{Periodicity of Betti numbers of monomial curves}, J. Algebra {\bf 418} (2014), 66--90.

\end{thebibliography}
\end{document}